\documentclass[11pt, a4paper]{article}

\usepackage{a4wide}
\usepackage{amsmath}
\usepackage{amssymb}
\usepackage{amsthm}
\usepackage{txfonts}
\usepackage{mathtools}
\usepackage{tikz}
\usepackage{mdwlist}
\usepackage{xcolor}
\usepackage{comment}
\usepackage{ifthen}
\usepackage{hyperref}

\newif\ifNODEF\NODEFfalse

\newtheorem{theorem}{Theorem}
\newtheorem{remark}{Remark}
\newtheorem{lemma}{Lemma}

\newtheorem{ass}{Assumption}

\renewcommand{\eqref}[1]{\hyperref[#1]{(\ref{#1})}}

\newcommand{\pt}{\partial}

\newcommand{\RR}{\mathbb{R}}
\renewcommand{\phi}{\varphi}
\renewcommand{\rho}{\varrho}
\renewcommand{\theta}{\vartheta}
\newcommand{\scal}[2]{\left(#1,#2\right)}
\newcommand{\dual}[2]{\left\langle#1,#2\right\rangle}
\newcommand{\bigscal}[2]{\bigl(#1,#2\bigr)}
\newcommand{\bigdual}[2]{\bigl\langle#1,#2\bigr\rangle}
\newcommand{\m}[1]{\mathcal{#1}}

\newcommand{\mK}{\mathcal{K}}
\newcommand{\mA}{\mathcal{A}}
\newcommand{\mL}{\mathcal{L}}
\newcommand{\mM}{\mathcal{M}}
\newcommand{\iL}{\mL^{-1}}
\newcommand{\iLM}{\mL^{-1}\mM}
\newcommand{\mMa}{\mM^*}

\newcommand{\iLMa}{\iL\mMa}

\newcommand{\Hp}[2]{H^1_0(#1) \cap H^#2(#1)}

\newcommand{\Z}{\mathbb{Z}}
\newcommand{\E}{\mathrm{e}}
\newcommand{\D}{\mathrm{d}}
\newcommand{\IntD}{\D}
\newcommand{\abs}[1]{\left|#1\right|}
\newcommand{\norm}[1]{\left\|#1\right\|}

\newcommand{\Opnorm}[1]{\left|\left|\left| #1\right|\right|\right|}	
\newcommand{\bignorm}[1]{\big\|#1\big\|}

\hypersetup{
  pdftitle    = {A posteriori error bounds for pseudo-parabolic equations
                 using $C_0$~semigroups},
  pdfsubject  = {Numerical Analysis},
  pdfauthor   = {Marin Ossadnik, Torsten Lin\ss},
  pdfkeywords = {pseudo-parabolic problems, a posteriori error estimation,
  elliptic reconstructions},
  pdfcreator  = {pdflatex},
  pdfproducer = {LaTeX with hyperref},
    colorlinks,
    linkcolor={red!90!black},
    citecolor={blue!90!black},
    urlcolor={green!90!black}
}

\author{Martin Ossadnik
   \and Torsten Lin\ss\thanks{Fakult\"at f\"ur Mathematik und Informatik,
        FernUniversit\"at in Hagen,
        Universit\"atsstra{\ss}e 11,
        58095 Hagen,
        Germany,
        \texttt{[martin.ossadnik,torsten.linss]@fernuni-hagen.de}}
}

\title{A posteriori error bounds for\\ pseudo-parabolic equations
       using $C_0$~semigroups}

\begin{document}
\maketitle
	
\begin{abstract}
  A class of pseudo-parabolic partial differential
  equations is considered.
  We derive a posteriori error bounds for approximations
  obtained by FEMs in space and a BDF formula in time.
  The analysis is based on the $C_0$~semigroup theory and an
  adaptation of the concept of elliptic reconstruction to
  pseudo-parabolic problems.
  The analysis is complemented with numerical experiments. 	
	
  \emph{Keywords:}
  pseudo-parabolic problems, a posteriori error estimation,
  elliptic reconstructions.

  \emph{AMS subject classification (2000):} 65M15, 65M50, 65M60.
\end{abstract}

\section{Introduction}

We consider the pseudo-parabolic problem of finding
\mbox{$u\colon[0,T] \mapsto H^1_0(\Omega)$},
\mbox{$\Omega \subset \RR^n$}, such that
\begin{subequations}\label{ibvp}
\begin{gather}
  \mK u(t) \coloneqq \mathcal{L} \partial_tu(t) + \mathcal{M} u(t) = F(t), \quad \text{in } (0,T]\, ,
\end{gather}
with two second-order, elliptic, time-independent operators
\mbox{$\mathcal{L}, \mathcal{M}\colon H^1_0(\Omega) \mapsto H^{-1}(\Omega)$}
and a source function \mbox{$F\in C \bigl([0,T]; H^{-1}(\Omega) \bigr)$}.
Furthermore an initial condition
\begin{gather}
  u(0) = u_0, \quad u_0 \in H^1_0(\Omega),
\end{gather}
is given.
We shall assume, that the operator $\mathcal{M}$ is bounded,
while $\mathcal{L}$ is bounded, coercive and symmetric.
By coercive we mean that the associated bilinear form is coercive.

\end{subequations}
	
Numerical approximations of pseudo-parabolic equations have been developed
in, e.\,g., \cite{ewing1975sobolev, ewing1978sobolev, ford1974sobolev,
ford1976sobolev, thomee1981sobolev}.
Most of these works focus on numerical methods for computing approximate
solutions and on \textit{a priori} estimates, i.\,e., convergence results are
derived in terms of the spatial and temporal mesh sizes under certain
regularity assumptions on the solution $u$.
In contrast, we focus on \textit{a posteriori} error bounds.
They are based on the computed numerical solution,
and therefore represent computable upper bounds for the error of the
numerical approximation.  

In \cite{tran2005estimation} an a posteriori error bound for semidiscrete
finite element methods for a nonlinear \textsc{Sobolev} equation has been
derived.
In contrast, we will present a posteriori error bounds
for full discretizations, i.\,e., both in time and in space.

An important contribution to the a posteriori error analysis of
\textit{parabolic problems} was the introduction of the concept of
\textit{elliptic reconstruction} by Nochetto and Makridakis
in~\cite{nochetto2003reconstrution}.
Subsequently, their ideas were used to analyze a number of discretizations
of parabolic problems by FEM and various time discretiszations methods,
including backward \textsc{Euler}, \textsc{Crank-Nicolson}, dG(1) and BDF-2
(cf. \cite{demlow2016error, demlow2010error, linss2013estimation,
ossadnik2024unified}).
We will adapt elliptic reconstruction to pseudo-parabolic problems
by leveraging the operator $\mL$. 

Another ingredient in the a posteriori error analysis in the contributions
mentioned is the use of the \textsc{Green}'s function associated with the
parabolic operator in order to obtain bounds in the maximum norm.
We will use a similar approach by considering the $C_0$~semigroup,
generated by the operator $-\iLM$,
combined with the mentioned adaptation of elliptic reconstructions.
As a result, we obtain computable error bounds in
\mbox{$C\bigl([0,T]; H^1_0(\Omega)\bigr)$} and 
\mbox{$C\bigl([0,T]; L^2(\Omega)\bigr)$} for time discretization methods 
of second order.
	
The paper is structured as follows.
In \autoref{sec-pre} we introduce some notation and then discuss the
existence of a unique solution to \eqref{ibvp}.
Furthermore, we present an estimate for the operator norms of the operators
of the $C_0$~semigroup, generated by the operator $\iLM$.
As mentioned, we will adapt elliptic reconstruction to
pseudo-parabolic equations.
This is done in \autoref{sec-ell-rec}.
Afterwards, in \autoref{ssec-green}, we introduce a function,
that is similar to the \textsc{Green}'s function,
to derive a \mbox{$C\bigl([0,T]; L^2(\Omega)\bigr)$} a posteriori error bound.
Our main results, \hyperref[theo-h1-est]{Theorems~\ref{theo-h1-est}}
and \ref{theo-l2-est},
will be derived in \hyperref[sec-h1-error]{Sections~\ref{sec-h1-error}}
and~\ref{sec-l2-error}.
Finally, a numerical example is considered in \autoref{sec-num-ex}.

\section{Preliminaries}\label{sec-pre}

\subsection{Notation}\label{subsec-not}
Let \mbox{$\scal{\cdot}{\cdot}$} be the standard inner product in $L_2(\Omega)$
and $\norm{\cdot}_{0,\Omega}$ the norm induced.
We shall also use the \textsc{Sobolev} spaces $H^k(\Omega)$ equipped with
the standard norm $\norm{\cdot}_{k,\Omega}$, and the
space $H^1_0(\Omega)$.
Furthermore, we denote by
\mbox{$\dual{\cdot}{\cdot}\colon H^{-1}(\Omega) \times H^1_0(\Omega)$}
the duality pairing between $H^1_0(\Omega)$ and its dual space $H^{-1}(\Omega)$.
The norm on $H^{-1}(\Omega)$ will be denoted by $\norm{\cdot}_{-1,\Omega}$.

Let \mbox{$a, c \colon H^1_0(\Omega) \times H^1_0(\Omega) \mapsto \RR$}
be the two bilinear forms associated with the operators $\mL$ and $\mM$,
respectively.
Then~\eqref{ibvp} reads:
Find \mbox{$u\colon[0,T] \mapsto H^1_0(\Omega)$},
such that $u(0) = u_0$ and
\begin{gather}\label{ibvp-vari}
  a\bigscal{\partial_t u(t)}{v} + c\bigscal{u(t)}{v} = \bigdual{F(t)}{v},
     \ \ \forall \ v\in H_0^1(\Omega), \ \ t\in(0,T]. 
\end{gather}
We assume there exist positive constants $C_\mL$, $C_\mM$ and $\alpha$,
such that for all \mbox{$w,v \in H^1_0(\Omega)$}
\begin{gather}\label{ibvp-ass}
  \abs{a\scal{w}{v}} \leq C_\mL \norm{w}_{1,\Omega} \norm{v}_{1,\Omega}\, ,
  \quad
  \abs{c\scal{w}{v}} \leq C_\mM \norm{w}_{1,\Omega} \norm{v}_{1,\Omega}
  \quad \text{and} \quad
  a\scal{w}{w} \geq \alpha \norm{w}^2_{1,\Omega}. 
\end{gather}
Note that, due to the boundedness of $c$, there is a constant
\(\gamma \in \RR \), such that
\begin{gather*}
  c\scal{w}{w} \geq \gamma \norm{w}^2_{1,\Omega},
  \quad \forall w \in H^1_0(\Omega).
\end{gather*}
Moreover, $a\scal\cdot\cdot$ defines an inner product on $H^1_0(\Omega)$
and induces a norm $\norm{\cdot}_{a,\Omega}$,
which is equivalent to~$\norm{\cdot}_{1,\Omega}$. 
For any bounded operator
\mbox{\(\mA\colon H^l(\Omega) \to H^{m}(\Omega)\)}, \(l,m \in \Z\),
its operator norm is defined by
\begin{gather*}
  \Opnorm{\mA}_{l,m}
    \coloneqq \sup_{\norm{v}_{l, \Omega} = 1} \norm{\mA v}_{m,\Omega},
\end{gather*} 
If $l=m$, we write \mbox{$\Opnorm{\cdot}_l$} instead of $\Opnorm{\cdot}_{l,l}$.

Next, we introduce the spaces
\begin{gather*}
  L^2\bigl(0,T;H^1_0(\Omega)\bigr) \coloneqq
    \left\{v\colon [0,T] \mapsto H_0^1(\Omega)\colon
       \int_{0}^{T} \norm{u(t)}^2_{1,\Omega} \D t < \infty\right\}.
  \intertext{and}
  H^1\bigl(0,T; H^1_0(\Omega)\bigr) \coloneqq
    \left\{v \in L_2\bigl(0,T; H^1_0(\Omega)\bigr)\colon
       v^\prime \in L_2\bigl(0,T;H^1_0(\Omega)\bigr)\right\}.
\end{gather*}

\subsection{Unique solvability and a \boldmath$C_0$~semigroup estimate}
\label{subsec-est-semigroup}

First, since \mbox{$a\scal\cdot\cdot$} is bounded and coercive
on $H^1_0(\Omega)$, an application of the \textsc{Lax-Milgram} Lemma
implies, that $\mL$ possesses a bounded
inverse \mbox{$\iL\colon H^{-1}(\Omega) \mapsto H^1_0(\Omega)$}.
Therefore, \eqref{ibvp} can equivalently be rewritten as
\begin{subequations}\label{ibvp-transform}
  \begin{align}
    \label{ibvp-transform-de}
    \pt_t u(t) &= -\iLM u(t) + \iL F(t), \quad \text{for } t \in (0,T], \\
    \label{ibvp-transform-iv}
          u(0) &= u_0.
  \end{align}	
\end{subequations} 
The \textsc{Lax-Milgram} Lemma and~\eqref{ibvp-ass} imply
\mbox{$\Opnorm{\iL}_{-1,1} \le 1/\alpha$}.
Furthermore, \mbox{$\Opnorm{\mM}_{1,-1} \le C_\mM$}.
As a consequence, the operator $-\iLM$ is bounded with
\begin{gather}\label{eq-bound-lm}
  \Opnorm{\iLM}_1 \leq \Opnorm{\iL}_{-1,1} \Opnorm{\mM}_{1,-1} 
                  \leq \frac{C_\mM}{\alpha}
\end{gather}
and thus the generator of a uniformly continuous
semigroup \mbox{$\bigl(S(t)\bigr)_{t\geq 0}$},
cf. \cite[Theorem 1.1.2]{pazy1983semigroups}.
Furthermore, we can conclude, that \eqref{ibvp} possesses a unique solution
\mbox{$u \in C^1\bigl([0,T]; H^1_0(\Omega)\bigr)$}
cf. \cite[Theorem 3.2.3]{tanabe1979evolution}.
Setting
\begin{gather}\label{S-exp}
  S(t) \coloneqq \E^{-t \iLM}, \quad t\geq 0,
\end{gather}
its solution $u$ can be represented as
\begin{gather} \label{eq-sol-u}
  u(t) = S(t)u_0 + \int_0^t S(t-s) \iL F(s) \IntD s \quad \forall t\in[0,T].
\end{gather}	
\begin{lemma}\label{lem-bound-c0}
  The $C_0$~semigroup $(S(t))_{t\geq0}$
  generated by the operator $-\iLM$, satisfies
  \begin{gather}\label{eq-bound-c0}
    \Opnorm{S(t)}_1 \leq \eta_{S,1}(t) \coloneqq 
	\sqrt{\frac{C_\mL}{\alpha}} \E^{-\nu t},
        \quad \forall \ t\geq0, \quad \text{with}
        \quad \nu \coloneqq \min\left\{\frac{\gamma}{C_\mL},
                                       \frac{\gamma}{\alpha}\right\}.
  \end{gather}
\end{lemma}
\begin{proof}
  If \mbox{$\gamma > 0$}, then
  \begin{gather}\label{eq-bound-c0-gamma-greater}
    \Opnorm{S(t)}_1 \leq \sqrt{\frac{C_\mL}{\alpha}}
      \E^{-\frac{\gamma}{C_\mL}t}, \quad t\geq0, \quad
      \text{see \cite{ting1970sobolev}}.
  \end{gather}

  If \mbox{$\gamma \leq 0$} we use the techniques
  presented in \cite{ting1969sobolev}. 
  Consider the \textsc{Hilbert} space
  \mbox{$V_a \coloneqq \bigl(H^1_0(\Omega), a\scal{\cdot}{\cdot}\bigr)$}.
  The norms \mbox{$\norm{\cdot}_{1, \Omega}$}
  and \mbox{$\norm{\cdot}_{a,\Omega}$} are equivalent.
  Therefore, the operator $-\iLM$ is bounded and closed with respect
  to $\Opnorm{\cdot}_{a}$.
  Now, let $\lambda\in\RR$ be arbitrary,
  but fixed with \mbox{$\lambda > -\gamma/\alpha$}.
  For any given \mbox{$g\in V_a$}, consider the problem of finding
  \mbox{$w \in V_a$} such that
  \begin{gather*}
    \left(\lambda + \iLM\right)w = g \quad \text{or equivalently} \quad	\left(\lambda \mL + \mM\right)w = \mL g.
  \end{gather*}
  The bilinear form
  \mbox{$\lambda a\scal{\cdot}{\cdot} + c\scal{\cdot}{\cdot}$}
  is both bounded and coercive with coercivity constant
  \mbox{$\lambda + \gamma/\alpha>0$}.
  Consequently, the operator
  \mbox{$\left(\lambda + \iLM\right)\colon V_a \mapsto V_a$} is invertible
  and we obtain by means of the
  \textsc{Lax-Milgram} Lemma and by the \textsc{Fréchet-Riesz} Theorem,
  \begin{gather*}
    \norm{\left(\lambda + \iLM\right)^{-1}g}_{a,\Omega}
      = \norm{\left(\lambda \mL + \mM\right)^{-1}\mL g}_{a,\Omega}
      \leq \frac{\norm{\mL g}_{-a,\Omega}}{\lambda + \gamma/\alpha }
      = \frac{\norm{g}_{a,\Omega}}{\lambda + \gamma/\alpha},
  \end{gather*}
  where $\norm{\cdot}_{-a,\Omega}$ is the norm on $V^*_a$, the dual of $V_a$.
  Next, since \mbox{$-\iLM$} is closed and densely defined,
  we can apply the \textsc{Hille-Yosida} Theorem
  \cite[Corollary 2.3.6]{engel2000semigroups}, concluding that 
  \begin{gather*}
    \Opnorm{S(t)}_a \leq \E^{-\frac{\gamma}{\alpha}t}, \quad t\geq0.
  \end{gather*}
  It follows that
  \begin{gather}\label{eq-bound-c0-gamma-lower}
    \Opnorm{S(t)}_1 \leq \sqrt{\frac{C_\mL}{\alpha}}
                         \E^{-\frac{\gamma}{\alpha}t}, \quad t\geq0,
  \end{gather}
  because \mbox{$\alpha \norm{\cdot}^2_{1,\Omega}
                   \leq \norm{\cdot}^2_{a,\Omega}
                   \leq C_\mL \norm{\cdot}^2_{1,\Omega}$}.

  Combining the bounds~\eqref{eq-bound-c0-gamma-greater} and~\eqref{eq-bound-c0-gamma-lower},
  we obtain the proposition of the lemma.
\end{proof}

For every \mbox{$\phi \in L^p\bigl(0,T; H^1_0(\Omega)\bigr)$} one has
\mbox{$S(T-\cdot)\,\phi(\cdot) \in L^p(0,T;H^1_0(\Omega))$},
cf. \cite[Prop. C.2, C.3]{engel2000semigroups}.
Therefore, the inequality
\begin{gather*}
  \norm{\int_{0}^{T} S(T-s)\phi(s) \IntD s}_{1,\Omega}
    \leq \int_{0}^{T} \norm{S(T-s)\phi(s)}_{1,\Omega} \D s
\end{gather*} 
holds (cf. \cite[Theorem II.4]{diestel1977measures}).
From this inequality, \autoref{lem-bound-c0} and the \textsc{Lax-Milgram}
Lemma, we get the following a priori estimate for the solution $u$
of~\eqref{ibvp}:
\begin{gather*}
  \norm{u(t)}_{1,\Omega}
    \leq \eta_{S,1}(t) \norm{u_0}_{1,\Omega}
            + \frac{1}{\alpha} \int_{0}^{t} \eta_{S,1}(T-s)
                                            \norm{F(s)}_{-1,\Omega} \IntD s,
  \quad \forall t\in[0,T].
\end{gather*}

If the data is smoother, we can expect higher regularity for 
the solution \(u\) too.
To be precise, if the domain \(\Omega\) and the coefficients 
of \(\mL\) and \(\mM\) are sufficiently regular, then the operator 
\(\iLM\) is a bounded operator on \(\Hp{\Omega}{2}\) (cf. \cite{ting1970sobolev}). 
Consequently, given \mbox{$u_0 \in \Hp{\Omega}{2}$} and 
\mbox{$F \in C([0,T]; L^2(\Omega))$}
the solution of \eqref{ibvp} satisfies 
\mbox{$u \in C^1\left([0,T]; \in H^1_0(\Omega) \cap H^2(\Omega)\right)$} 
(cf. \cite[Theorem~3.2.3]{tanabe1979evolution}). 

\subsection{An \boldmath$L^2$-norm representation}
\label{ssec-green}

In \autoref{sec-l2-error} we shall derive an a posteriori error bound for
the error at final time $T$ measured in the $L_2(\Omega)$-norm.
The central ingredient is \autoref{lem-green-l2} below.
For any function
\mbox{$v \in H^1\left(0,T;H^1_0(\Omega)\right),$} it provides
a representation of $\norm{v(T)}_{0,\Omega}$ in terms of $\mK v$ and of
its initial value $v(0)$.
The analysis is carried out under the following regularity assumptions. 
\begin{ass}\label{ass-coeffs-l-m}
  Let the coefficients of $\mM$ be sufficiently smooth,
  such that \mbox{$\mM v,\,\mM^* v \in L^2(\Omega)$}
  for all \mbox{$v \in H^2(\Omega)$}.
  Furthermore, assume that the domain $\Omega$ and the coefficients of $\mL$
  are such that for any \mbox{$g \in L^2(\Omega)$},
  the solution \mbox{$y \in H^1_0(\Omega)$} of the elliptic boundary value
  problem \mbox{$\mL y = g$} satisfies \mbox{$u \in H^2(\Omega)$}.
\end{ass} 
\begin{remark}
  The assumption on the domain are for example satisfied if $\Omega$ is convex,
  see \cite{grisvard1985elliptic}.
\end{remark}
For arbitrary \mbox{$\chi\in L^2(\Omega)$} let
\mbox{$G_\chi\colon [0,\infty) \mapsto H^1_0(\Omega)$} be the solution of
the initial value problem
\begin{subequations}\label{eq-green-like}
  \begin{align}
    \mL \pt_tG_\chi(s) + \mMa G_\chi(s) & = 0, \quad s\in [0,\infty)
      \label{eq-green-pde}\\
                         G_\chi(0) & = \iL \chi.
      \label{eq-green-init}
  \end{align}
\end{subequations}
\autoref{ass-coeffs-l-m} guaranties that \mbox{$\iL\chi \in H^2(\Omega)$}
and that \mbox{$\mMa v \in L^2(\Omega) \ \ \forall \ v \in H^2(\Omega)$}.
Consequently, the results of \autoref{subsec-est-semigroup} can be applied,
and there exists
a unique solution \mbox{$G_\chi\in C^1\bigl([0,\infty);\Hp{\Omega}{2}\bigr)$}
of~\eqref{eq-green-like}:
\begin{gather}
  G_\chi(s) = S^*(s) \iL \chi, \quad s\geq 0, \quad
  \text{with} \ \ S^*(s)  \coloneqq \E^{-s\iLMa}.
\end{gather}

\begin{lemma}\label{lem-green-l2}
  For all \mbox{$v \in H^1\bigl((0,T);H_0^1(\Omega)\bigr)$}, we have the
  representation
  \begin{gather*}
    \norm{v(t)}^2_{0, \Omega} = a\scal{G_{v(t)}(t)}{v(0)}
       + \int_{0}^{t} \dual{\mK v(s)}{G_{v(t)}(t-s)} \D s.
  \end{gather*}
\end{lemma}
\begin{proof}
  Integrating \eqref{eq-green-pde}, we obtain
  \begin{gather*}
    0 = \int_0^t \dual{\mL \pt_s G_{v(t)}(t-s)}{v(s)} \D s
          + \int_{0}^{t} \dual{\mMa G_{v(t)}(t-s)}{v(s)} \D s
  \end{gather*}
  Then, employing \eqref{eq-green-init} and integration by parts
  yields
  \begin{align*}
    \norm{v(t)}^2_{0,\Omega}
      & = \dual{v(t)}{v(t)} = \dual{\mL G_{v(t)}(0)}{v(t)} \\
      & = \dual{\mL G_{v(t)}(t)}{v(0)}
             + \int_{0}^{t} \dual{\mL \pt_s v(s)}{G_{v(t)}(t-s)} \IntD s
             + \int_{0}^{t} \dual{\mM v(s)}{G_{v(t)}(t-s)} \D s \\ 
      & = a\scal{G_{v(t)}(t)}{v(0)}
             + \int_{0}^{t} \dual{\mK v(s)}{G_{v(t)}(t-s)} \D s.
  \end{align*} 

  \vspace*{-4ex}
\end{proof}
When deriving our a posteriori error estimator in \autoref{sec-l2-error},
the following bounds for \mbox{$\norm{G_\chi(s)}_{1,\Omega}$}
and \mbox{$\norm{G_\chi(s)}_{2,\Omega}$} will be used.
\begin{lemma}\label{lem:green:estimates}
  There are constants \mbox{$M_2, \omega_2 \in \RR$},
  \mbox{$M_2 > 0$}, such that
  \begin{align}
    \norm{G_\chi(s)}_{1,\Omega}
      & \leq \eta_{S,1}(t) 
             \frac{\norm{\chi}_{0,\Omega}}{\alpha}
             \quad \forall s \in [0,\infty) \label{eq:G-1}\\
    \intertext{and}
    \norm{G_\chi(s)}_{2,\Omega}
      & \leq \eta_{S,2}(s)
  	     \Opnorm{\iL}_{0,2} \norm{\chi}_{0,\Omega}
             \quad \forall s \in [0,\infty) \label{eq:G-2}
  \end{align}
  with $\eta_{S,1}$ from \autoref{lem-bound-c0} and
  \begin{gather*}
    \eta_{S,2}(s) \coloneqq M_2 \E^{t\omega_2}
         \quad \forall t \in [0,\infty).
  \end{gather*}
\end{lemma}
\begin{proof}
  Owing to the definition of \(G_\chi\), we have
  \begin{gather*}
    \norm{G_\chi(s)}_{1,\Omega}
      = \norm{S^*(s) \iL \chi}_{1,\Omega}
     \leq \Opnorm{S^*(s)}_1 \norm{\iL \chi}_{1,\Omega}\,.
  \end{gather*}
  Since \mbox{\(L^2(\Omega) \subset H^{-1}(\Omega)\)} -- in the sense of
  \textsc{Gelfand} triples -- we infer from the \textsc{Lax-Milgram} Lemma
  that
  \begin{gather*}
    \norm{\iL \chi}_{1,\Omega} \leq \frac{\norm{\chi}_{0,\Omega}}{\alpha}\,.
  \end{gather*}
  To prove \eqref{eq:G-1}, it suffices to derive an upper bound on
  \mbox{$\Opnorm{S^*(t)}_1$}.
  Since \mbox{$\dual{\mMa v}{v} = \dual{\mM v}{v} = c\scal{v}{v}$}
  for all \mbox{$v\in H^1_0(\Omega)$}
  and because \mbox{$\Opnorm{\mM^*}_{1,-1} = \Opnorm{\mM}_{1,-1}$},
  the bilinear form associated with the operator
  \mbox{\(\lambda \mL + \mMa\colon H^1_0(\Omega) \to H^{-1}(\Omega)\)}
  is bounded and coercive for all
  \mbox{$\lambda > -\min\left\{\gamma/C_\mL, \gamma/\alpha\right\}$}.
  Proceeding as in \autoref{subsec-est-semigroup}, we conclude that
  \begin{gather}\label{eq-c0-bound-green}
    \Opnorm{S^*(t)}_1 \leq \eta_{S,1}(t), \quad \forall t \in [0,\infty),
  \end{gather}
  with $\eta_{S,1}$ from \autoref{lem-bound-c0}.
  
  Bounding $\norm{G_\chi(s)}_{2,\Omega}$ is slightly more delicate.
  The boundedness of \mbox{\(\mL\)} and \autoref{ass-coeffs-l-m} imply that
  the restriction of \(\mL\) onto \(\Hp{\Omega}{2}\) is a bounded and
  bijective operator from 
  \mbox{\(\Hp{\Omega}{2}\)} onto \(L^2(\Omega)\).
  Thus, an application of the Bounded Inverse Theorem yields that the
  restriction of \(\iL\) onto \(L^2(\Omega)\) is bounded operator
  onto \(\Hp{\Omega}{2}\).
  Furthermore, the operators 
  \begin{gather*}
     S(t)\colon \Hp{\Omega}{2} \to \Hp{\Omega}{2}, \quad t\ge 0,
  \end{gather*}
  form a \(C_0\)~semigroup on \(\Hp{\Omega}{2}\), cf. \cite{ting1970sobolev}.
  Hence, there exist constants \(M_2>0\) and \mbox{\(\omega_2 \in \RR\)},
  such that -- see \cite[Proposition I.5.5]{engel2000semigroups} --
  \begin{gather*}
    \Opnorm{S^*(s)}_{2} \leq \eta_{S,2}(s) \coloneqq M_2 \E^{\omega_2s},
      \quad s\geq 0.
  \end{gather*} 
  We infer that
  \begin{align*}
    \norm{G_\chi(s)}_{2,\Omega} = \norm{S^*(s)\iL \chi}_{2,\Omega}
      \leq \norm{S^*(s)}_2 \norm{\iL \chi}_{2,\Omega}
      \leq \eta_{S,2}(s) \Opnorm{\iL}_{0,2} \norm{\chi}_{0,\Omega}	
  \end{align*}

  \vspace*{-2ex}
\end{proof}

\section{Elliptic reconstruction for pseudo-parabolic problems}
\label{sec-ell-rec}

The concept of \emph{elliptic reconstruction} was introduced 
in \cite{nochetto2003reconstrution}.
It allows the use of a posteriori error analysis for elliptic problems
as a building block in the derivation of a posterior error estimators for
parabolic problems.
In that paper -- and a number of follow up publications -- the operator
$\m{M}$ was used to define the elliptic recontruction.
However, in the pseudo-parabolic problem~\eqref{ibvp} the operator $\m{M}$
lacks coercivity and is therefore unsuitable to define the reconstruction.
Instead, we shall employ the operator $\mL$ in order to make elliptic
reconstruction applicable to pseudo-parabolic problems. 

\subsection{A posteriori error bound for elliptic problems}
\label{ssect-apost-ell}

Given \mbox{$G\in H^{-1}(\Omega)$}, consider the elliptic boundary-value
problem of finding \mbox{$y \in H^1_0(\Omega)$} such that
\begin{gather}\label{eq-ell}
  a\scal{y}{v} = \dual{G}{v} \ \ \forall v \in H^1_0(\Omega).
\end{gather}
Next, given a finite dimensional subspace $V^0_h$ of $H_0^1(\Omega)$ we
seek an approximation \mbox{$y_h \in V^0_h$} of $y$ as the solution of
\begin{gather}\label{eq-ell-disc}
  a_h\scal{y_h}{v_h} = \dual{G}{v_h}_h \ \ \forall v_h \in V^0_h,
\end{gather}
where $a_h\scal{\cdot}{\cdot}$ and $\dual{\cdot}{\cdot}_h$ are approximations
of $a\scal{\cdot}{\cdot}$ and $\dual{\cdot}{\cdot}$, resp.

Our a posteriori analysis for the pseudo-parabolic problem relies on the
following assumption.
\begin{ass}\label{ass-ell-est}
  Let $y$ and $y_h$ be the solutions of \eqref{eq-ell} and \eqref{eq-ell-disc},
  respectively.
  We assume there exists an a posteriori error estimator $\eta$,
  depending on $G$ and $y_h$ only, such that
  \begin{gather*}
    \norm{y-y_h}_{1,\Omega} \leq \eta\left(y_h, G\right).
  \end{gather*}
\end{ass}

\subsection{Elliptic reconstruction}

We consider arbitrary time stepping procedures for \eqref{ibvp}.
To this end, let $\omega_t$ be a mesh in time given by
\begin{gather*}
  \omega_t\colon 0 = t_0 < \cdots < t_M = T,
\end{gather*} 
with mesh intervals \mbox{$I_j \coloneqq (t_{j-1}, t_j)$} and
step sizes \mbox{$\tau_j \coloneqq t_j - t_{j-1}$}, \mbox{$j=1,\dots,M$}.
The midpoints of the mesh intervals are denoted by
\mbox{$t_{j-1/2}=\bigl(t_{j-1}+t_j\bigr)/2$}.
For any function \mbox{$v\colon \Omega \times \omega_t \mapsto \RR$}
we use the abbreviations \mbox{$v^j \coloneqq v(\cdot, t_j)$}
and \mbox{$v^{j-1/2} \coloneqq v(\cdot, t_{j-1/2})$},
\mbox{$j=1,\dots,M$}.
Furthermore, let
\begin{gather*}
  \delta_t v^j \coloneqq \frac{v^j - v^{j-1}}{\tau_j}, \quad j=1,\dots,M.
\end{gather*}
As mentioned in the beginning, the framework presented in the sequel,
is not specific to a particular time discretization.

Let again $V^0_h$ be a finite dimensional subspace of $H^1_0(\Omega)$, and let
\mbox{$u^0_h, u^1_h,\dots, u^M_h \in V^0_h$} be approximations
of \mbox{$u(t_0), u(t_1), \dots, u(t_M)$}.
We assume the latter are obtained be a finite element
method using the approximations
$a_h\scal{\cdot}{\cdot}$, $c_h\scal{\cdot}{\cdot}$,
$\scal{\cdot}{\cdot}_h$ and $\dual{\cdot}{\cdot}_h$ of the continuous
forms and pairings
$a\scal{\cdot}{\cdot}$, $c\scal{\cdot}{\cdot}$, $\scal{\cdot}{\cdot}$
and $\dual\cdot\cdot$, resp.

To derive our a posteriori error bounds, we need to extend these
pointwise approximations to a function $\tilde{u}_h$ defined on
the whole time interval $[0,T]$.
We shall use piecewise linear interpolation:
For any function $v$ defined on $\omega_t$, $t_j \mapsto v^j$, let
\begin{gather}\label{eq-int}
  \tilde{v}(s) \coloneqq \frac{v^j+v^{j-1}}{2} + (s-t_{j-1/2}) \delta_t v^j,
    \quad s \in \bar{I}_j, \quad j=1,\dots,M. 
\end{gather}
Due to this construction we
have \mbox{$\tilde{u}_h \in W^{1,2}\left(0,T; H^1_0(\Omega)\right)$}. 

Next, given an approximation \mbox{$u^j_h \in V^0_h$}
of \mbox{$u(t_j)\in H^1_0(\Omega)$},
we define \mbox{$\psi^j_h \in V^0_h$} by
\begin{gather}\label{eq-def-psi}
  a_h\scal{\psi_h^j}{v_h} = c_h\scal{u^j_h}{v_h} - \dual{F^j}{v_h}_h
      \quad \forall v_h \in V^0_h, \quad j=0,\dots,M.
\end{gather} 
In view of~\eqref{ibvp-vari}, $\psi_h^j$ may be regarded as an approximation
to $\pt_t u(t^j)$.
Eq.~\eqref{eq-def-psi} can be rewritten as follows:
\begin{gather}
  a_h\scal{u^j_h}{v_h} = (a_h + c_h)\scal{u^j_h}{v_h} - \dual{F^j}{v_h}_h
     - a_h\scal{\psi^j_h}{v_h} \quad \forall v_h \in V^0_h, \quad j=0,\dots,M.
\end{gather}
Now, let \mbox{$R^j \in H^1_0(\Omega)$} be
the solution of the elliptic problem
\begin{gather}\label{eq-ell-def}
  a\scal{R^j}{v} = \dual{\Phi^j}{v} 
     \quad \forall v \in H_0^1(\Omega), \quad j=0,\dots,M.
  \intertext{where the functional $\Phi^j\in H^{-1}(\Omega)$ is defined by}
    \notag
    \dual{\Phi^j}{v} = (a+c)\scal{u_h^j}{v} - \dual{F^j}{v} - a\scal{\psi^j_h}{v}
     \quad \forall v \in H_0^1(\Omega).
\end{gather}
In operator notation this reads
\begin{gather}\label{eq-ell-def-op}
  \mL R^j = \bigl(\mL + \mM\bigr)\, u_h^j - F^j - \mL \psi_h^j.
\end{gather}
Coercivity and boundedness of $a(\cdot,\cdot)$
imply the existence and uniqueness of \mbox{$R^j \in H^1_0(\Omega)$}, which
is referred to as the elliptic reconstruction of \mbox{$u^j_h \in V^0_h$}.
The latter can be regarded as the finite-element solution of $R^j$.
Therefore, the a posteriori error estimator $\eta$ from \autoref{ass-ell-est}
can be applied to obtain an upper bound for the difference between
$R^j$ and $u^j_h$:
\begin{subequations}
\begin{gather}\label{eq-erro-ell-rec}
  \norm{R^j - u^j_h}_{a,\Omega} \leq \eta_{\mathrm{ell}}^j
    \coloneqq \eta\left(u^j_h, \Phi^j\right), \quad j=0,\dots,M.
\end{gather}
Because of linearity and since the operators $\mL$ and $\mM$
are both independent of time, we also have 
\begin{gather}\label{eq-half-err-ell-rec}
  \norm{\frac{R^j + R^{j-1}}{2} -
        \frac{u^j_h + u^{j-1}_h}{2}}_{a,\Omega}
    \leq \eta_{\mathrm{ell}, 1/2}^j
    \coloneqq \eta\left(\frac{u^j_h + u^{j-1}_h}{2}, \frac{\Phi^j + \Phi^{j-1}}{2}
                  \right), \quad j=1,\dots,M,
\end{gather}
and
\begin{gather}\label{eq-delta-err-ell-rec}
  \norm{\delta_t\left(R^j - u^j_h\right)}_{a,\Omega}
    \leq \eta_{\mathrm{ell}, \delta_t}^j
    \coloneqq \eta\left(\delta_tu^j_h, \delta_t \Phi^j\right),
      \quad j=1,\dots,M.
\end{gather}
\end{subequations}
	
\section{A posteriori error bounds in \boldmath$H^1(\Omega)$}
\label{sec-h1-error}

In this section we will derive an a posteriori bound for
the error at final time \mbox{$T=t_M$}, i.\,e.,
\begin{gather*}
  \norm{u(T)-u_h^M}_{1,\Omega}.
\end{gather*}
To this end, let
\begin{gather}\label{eq-def-g}
  g(s) \coloneqq S(T-s)\left(u-\tilde{u}_h\right)(s) \quad s\in [0,T].
\end{gather} 
and note that $g(T)=u(T)-u_h^M$.
Owing to the properties of $\iLM$ and $(S(t))_{t\geq0}$,
the function $g$ is continuous on \mbox{$[0,T]$} and continuously differentiable
on \mbox{$[0,T]\setminus \omega_t$} with
\begin{gather*}
  g' (s) = \iLM \: S(T-s) \left(u - \tilde{u}_h\right)(s)
              + S(T-s)\pt_s \left(u - \tilde{u}_h\right)(s),
      \quad s\in [0,T] \setminus \omega_t.
\end{gather*}
We use~\eqref{ibvp-transform-de} to eliminate $\pt_s u$, integrate over
$(0,T)$ and arrive at
\begin{gather*}
  g(T) - g(0) = \int_0^T S(T-s)
     \left(\iL F(s) - \iLM \tilde{u}_h(s) - \pt_s \tilde{u}_h(s)\right) \IntD s.
\end{gather*}
Next, \eqref{eq-def-g} gives
\begin{gather*}
  \bigl(u-u_h\bigr)(T)
     = S(T)\left(u_0-u^0_h\right)
       + \int_{0}^T S(T-s) \left( \iL F - \pt_s\tilde{u}_h - \iLM \tilde{u}_h
                           \right)(s) \, \D s.
\end{gather*}
Eq.~\eqref{eq-ell-def-op} implies that
$ \iL \mM \tilde{u}_h = R^j - \tilde{u}_h + \iL \tilde{F} + \tilde{\psi}_h $.
Therefore,
\begin{gather}\label{eq-u-tild-u-T-Euler}
  \begin{split}
    \bigl(u-u_h\bigr)(T)
       &= S(T)\left(u_0-u^0_h\right)
              + \int_{0}^T S(T-s)\iL\left(F-\tilde{F}\right)(s) \, \D s  \\
       & \qquad - \int_{0}^T S(T-s) \left(\tilde{R}-\tilde{u}_h
                                    \right)(s) \, \D s 
                - \int_{0}^T S(T-s) \left(\tilde{\psi}_h(s) + \pt_s\tilde{u}_h(s)
                                    \right) \D s,
  \end{split}
\end{gather}
where the integrals have to be understood in the \textsc{Bochner} sense.

\begin{theorem}\label{theo-h1-est}
  Let $u$ be the solution of \eqref{ibvp} and let \mbox{$u^0_h, \dots, u^M_h \in V^0_h$}
  be a sequence of approximations of \mbox{$u(t_0), \dots, u(t_M)$}.
  Then
  \begin{gather}
    \norm{u(T) - u^M_h}_{1,\Omega} \leq \eta_{H^1}
      \coloneqq \eta_{\mathrm{init}} + \eta_F + \eta_{\Psi}
                   + \eta_{\delta \psi} + \eta_{\mathrm{ell}},
  \end{gather}
  where
  \begin{gather*}
    \eta_{\mathrm{init}}
      \coloneqq \eta_{S,1}(T) \norm{u_0-u^0_h}_{1,\Omega}, \quad
    \eta_F \coloneqq \sum_{j=1}^{M} \frac{\sigma_j}{\alpha} \int_{I_j} \norm{\left(F-\tilde{F}\right)(s)}_{-1,\Omega} \D s \\
    \eta_{\Psi}
      \coloneqq \sum_{j=1}^{M} \sigma_j \tau_j \norm{\frac{\psi^j_h + \psi^{j-1}_h}{2} + \delta_tu^j_h}_{1,\Omega}, \quad
    \eta_{\delta \psi}
      \coloneqq \sum_{j=1}^{M} \sigma_j \chi_j \norm{\delta_t \psi^j_h}_{1,\Omega}
    \intertext{and}
    \eta_{\mathrm{ell}} \coloneqq \sum_{j=1}^{M} \sigma_j \left(\tau_j \eta_{\mathrm{ell}, 1/2}^j
                + \frac{\tau_j^2}{2} \eta_{\mathrm{ell}, \delta_t}^j\right).
  \end{gather*}
  with $\eta_{S,1}$ as defined in \autoref{lem-bound-c0},
  \begin{gather*}
    \sigma_j \coloneqq \max_{s\in [t_{j-1}, t_j]} \eta_{S,1}(T-s) \quad \text{and} \quad \chi_j \coloneqq \min\left\{\frac{\tau_j^2}{4}, \frac{C_\mM}{\alpha} \frac{\tau_j^3}{12}\right\}.
  \end{gather*}
\end{theorem}
\begin{proof}
  We start from \eqref{eq-u-tild-u-T-Euler} and bound the terms on its
  right-hand side individually.

  \noindent\emph{(i)} \ Application of \autoref{lem-bound-c0} gives
  \begin{gather*}
    \norm{S(T)\left(u_0 - u^0_h\right)}_{1,\Omega} \leq \eta_{S,1}(T) \norm{u_0-u^0_h}_{1,\Omega}.
  \end{gather*}

  \noindent\emph{(ii)} \
  Next,
  \begin{align*}
    & \norm{\int_{0}^{T}S(T-s)\iL\left(F-\tilde{F}\right)(s) \IntD s}_{1,\Omega}
      \le \int_{0}^{T} \Opnorm{S(T-s)}_1 \norm{\iL\left(F-\tilde{F}\right)(s)}_{1,\Omega} \IntD s \\
    & \qquad
      \le \sum_{j=1}^{M} \sigma_j \int_{I_j} \norm{\iL\left(F-\tilde{F}\right)(s)}_{1,\Omega} \IntD s.
  \end{align*}

  \noindent\emph{(iii)} \ We start similar to \emph{(ii)}:
  \begin{align*}
    & \norm{\int_{0}^{T}S(T-s)\left(\tilde{R}-\tilde{u}_h\right)(s) \IntD s}_{1,\Omega}
      \le \sum_{j=1}^{M} \int_{I_j} \eta_{S,1}(T-s) \norm{\left(\tilde{R}-\tilde{u}_h\right)(s)}_{1,\Omega} \IntD s \\
    & \qquad \leq \sum_{j=1}^{M} \sigma_j \int_{I_j}
            \norm{\frac{R^j - u^j_h + R^{j-1} - u^{j-1}_h}{2}+ (s-t_{j-1/2})\delta_t\left(R^j-u^j_h\right)}_{1,\Omega} \IntD s \\
    & \qquad \leq \sum_{j=1}^{M} \sigma_j \left(\int_{I_j} \norm{\frac{R^j - u^j_h + R^{j-1} - u^{j-1}_h}{2}}_{1,\Omega} \IntD s + \int_{I_j} \left|s-t_{j-1/2}\right| \norm{\delta_t\left(R^j - u^j_h\right)}_{1,\Omega} \IntD s \right).
  \end{align*}
  Utilizing \eqref{eq-erro-ell-rec} and \eqref{eq-delta-err-ell-rec}, gives
  \begin{align*}
    \norm{\int_{0}^{T}S(T-s)\iL\left(\tilde{R}-\tilde{u}_h\right)(s) \IntD s}_{1,\Omega}
      \leq \sum_{j=1}^{M} \sigma_j \left(\tau_j \eta_{\mathrm{ell}, 1/2}^j
                     + \frac{\tau_j^2}{4} \eta_{\mathrm{ell}, \delta_t}^j\right).
  \end{align*}

  \noindent\emph{(iv)} \
  For the last term in~\eqref{eq-u-tild-u-T-Euler} we have
  \begin{gather*}
    \tilde{\psi}_h(s) + \pt_s\tilde{u}_h(s)
      = \frac{\psi^j_h+\psi^{j-1}_h}{2} + \delta_t u^j_h
          + \left(s-t_{j-1/2}\right) \delta_t \psi^j_h.
  \end{gather*}
  Thus
  \begin{gather}\label{aux}
    \begin{split}
    & \norm{\int_{I_j} S(T-s) \left(\tilde{\psi}_h
              + \pt_s \tilde{u}_h\right)(s) \IntD s}_{1,\Omega} \\
    & \qquad
      \le \Opnorm{\int_{I_j} S(T-s)\: \D s}_1
          \norm{\frac{\psi^j_h+\psi^{j-1}_h}{2} + \delta_t u^j_h}_{1,\Omega}
          + \Opnorm{\int_{I_j} \bigl(s-t_{j-1/2}\bigr) \: S(T-s) \: \D s}_1
          \norm{\delta_t \psi^j_h}_1.
    \end{split}
  \end{gather}
  In view of~\eqref{eq-bound-c0} we have the bounds
  \begin{gather*}
    \Opnorm{\int_{I_j} S(T-s)\: \D s}_1 \le \tau_j \sigma_j
    \quad\text{and}\quad
    \Opnorm{\int_{I_j} \bigl(s-t_{j-1/2}\bigr) \: S(T-s) \: \D s}_1
      \le \frac{\tau_j^2}{4} \sigma_j.
  \end{gather*}
  Integration by parts and~\eqref{S-exp} yield
  \begin{align*}
    \int_{I_j} \bigl(s-t_{j-1/2}\bigr) \: S(T-s) \: \D s
      = \int_{I_j} \iLM \: \frac{(s-t_j)(s-t_{j-1})}{2} \: S(T-s)\: \D s.
  \end{align*}
  This together with \eqref{eq-bound-lm} gives the alternative bound
  \begin{gather*}
    \Opnorm{\int_{I_j} \bigl(s-t_{j-1/2}\bigr) \: S(T-s) \: \D s}_1
      \le \Opnorm{\iLM}_1 \frac{\tau_j^3}{6} \sigma_j
      \le \frac{C_\mM}{\alpha} \frac{\tau_j^3}{6} \sigma_j\,.
  \end{gather*}
  Thus
  \begin{gather*}
    \Opnorm{\int_{I_j} \bigl(s-t_{j-1/2}\bigr) \: S(T-s) \: \D s}_1
      \le \min\left\{\frac{\tau_j}{4}, \frac{C_\mM}{\alpha} \frac{\tau_j^3}{6}\right\}
              \sigma_j\,.
  \end{gather*}
  Applying this bound to \eqref{aux} and summing for $j=1,\dots,M$,
  we have bounded the last term in~\eqref{eq-u-tild-u-T-Euler}.
\end{proof}

\begin{remark}\label{rem:eta-f}
  The term $\eta_F$ involves the data of our problem~\eqref{ibvp} and
  inevitably needs to be approximated.
  For example one can use the \textsc{Simpson} rule to estimate
  \begin{gather*}
    \int_{I_j} \norm{\left(F-\tilde{F}\right)(s)}_{-1,\Omega} \D s
       \approx \frac{2\tau_j}{3} \norm{\left(F-\tilde{F}\right)(t_{j-1/2})}_{-1,\Omega}
       = \frac{\tau_j}{3} \norm{F^{j-1} - 2 F^{j-1/2} + F^j}_{-1,\Omega}
  \end{gather*}
\end{remark}

\section{A posteriori error bounds in \boldmath$L^2(\Omega)$}
\label{sec-l2-error}

It is well-known that the convergence of conforming finite element methods
in the $L^2$-norm is typically one order of $h$ faster then in the $H^1$-norm.
This asks for an alternative bound on $\norm{u(T)-u_h^M}_{0,\Omega}$ than
that induced by \autoref{theo-h1-est}.

To establish such a result, we leverage the function $G_\chi$,
introduced in \autoref{ssec-green} and our results from
\autoref{subsec-est-semigroup}.
We restrict ourselves to \mbox{$n\leq3$} in the following.
In particular, this guarantees the continuity of $G_\chi$ with respect to the
spatial variable.
However, we have to restrict the scope of our analysis and make further
assumptions.
We assume exact integration,
i.\,e., \mbox{$a_h=a$}, \mbox{$c_h=c$},
\mbox{$\scal{\cdot}{\cdot}_h=\scal{\cdot}{\cdot}$}
and \mbox{$\dual{\cdot}{\cdot}_h=\dual{\cdot}{\cdot}$}.
As a result, the $u^j_h$-s are the \textsc{Galerkin} approximations
of the $R^j$-s, \mbox{$j=0,\dots,M$}.
Furthermore, let $V^0_h$ be a finite-element space,
which is based on a conforming mesh $\mathcal{T}$ of $\Omega$.
Let $\m{T}$ belong to a shape-regular family of tessilations of $\Omega$.
The maximum diameter of all elements in $\m{T}$ is denoted
by $h_{\mathrm{max}}$.

For simplicity of notation set \mbox{\(G \coloneqq G_{(u-\tilde{u}_h)(T)}\)}.
Application of \autoref{lem-green-l2} to \mbox{$v=u-\tilde{u}_h$}, gives
\begin{gather*}
  \bignorm{\left(u-\tilde{u}_h\right)(T)}^2_{0,\Omega}
      = a\scal{G(T)}{u_0 - u^0_h }
           + \int_{0}^{T} \bigdual{\mK \left(u-\tilde{u}_h\right)(s)}{G(T-s)} \ \D s
\end{gather*}
We proceed similar to \autoref{sec-h1-error}.
By \eqref{eq-ell-def},
\begin{gather}\label{eq-res-l2}
  \begin{split}
    \bignorm{\left(u-\tilde{u}_h\right)(T)}^2_{0,\Omega} 
      & = a\scal{G(T)}{u_0 - u^0_h}
            + \int_0^T \dual{\bigl(F - \tilde{F}\bigr)(s)}{G(T-s)} \D s \\ 
      & \quad\ -\int_0^T a\scal{\bigl(\tilde{R}-\tilde{u}_h\bigr)(s)}{G(T-s)} \D s
          - \int_0^T a\scal{\bigl( \tilde{\psi}_h + \pt_s \tilde{u}_h \bigr)(s)}{G(T-s)} \D s.
  \end{split}
\end{gather}
The terms on the right-hand side will be bounded individually to establish
our next main result.
To state (and prove) this result we shall use the following auxilliary
statement.
\begin{lemma}\label{lem:project}
  Let \mbox{$P_h\colon H^2(\Omega) \to V^0_h$} be a projection operator.
  Then there exists a constant $C_I=C_I(P_h)$
  \begin{gather*}
    \norm{v-P_hv}_{1,\Omega} \leq C_I \, h_{\max} \abs{v}_{2,\Omega} \ \ \forall v\in H^2(\Omega),
  \end{gather*}
  where $\abs{\,\cdot\,}_{2,\Omega}$ is the standard semi norm on $H^2(\Omega)$.
\end{lemma}
\begin{proof}
  This follows from \cite[Theorem 4.28]{grossmann2007numerical}.
\end{proof}

\begin{theorem}\label{theo-l2-est}
  Let \autoref{ass-coeffs-l-m} hold true.
  Let $u$ be the solution of \eqref{ibvp} and
  let \mbox{$u^0_h,\dots,u^M_h \in V^0_h$}
  be a sequence of approximations of \mbox{$u(t_0), \dots, u(t_M)$},
  where $u^0_h$ satisfies
  \begin{gather}\label{eq-det-u0h-l2}
    a\scal{u_0 - u^0_h}{v_h} = 0, \quad \forall v_h \in V^0_h.
  \end{gather}
  Then there holds
  \begin{gather}\label{eq-l2-est}
    \norm{u(T)-u^M_h}_{0,\Omega}
      \le \eta_{L^2} \coloneqq \eta_{\mathrm{init}} + \eta_F
            + \eta_{\mathrm{ell}} + \eta_\Psi + \eta_{\delta \psi}
  \end{gather}
  with the various components given by
  \begin{gather*}
    \eta_{\mathrm{init}} \coloneqq
      C^* h_{\max} \: \eta_{S,2}(T) \norm{u_0 - u^0_h}_{1,\Omega}\, , \quad
    \eta_F \coloneqq \sum_{j=1}^{M} \frac{\sigma_j}{\alpha}
                \int_{I_j} \norm{\left(F-\tilde{F}\right)(s)}_{-1,\Omega} \D s \\
    \eta_{\Psi} \coloneqq \sum_{j=1}^{M} C_\mL \frac{\sigma_j}{\alpha} \tau_j
            \norm{\frac{\psi^j_h + \psi^{j-1}_h}{2} + \delta_t u^j_h}_{1,\Omega}\,, \quad
    \eta_{\delta \psi} \coloneqq \sum_{j=1}^{M} C_\mL \frac{\sigma_j}{\alpha} \chi_j \norm{\delta_t \psi^j_h}_{1,\Omega} \\
  \intertext{and}
    \eta_{\mathrm{ell}} \coloneqq C^* h_{\mathrm{max}}\sum_{j=1}^{M}
              \max_{s\in[t_{j-1}, t_j]}
                  \mu_j \left(\tau_j \eta^j_{\mathrm{ell}, 1/2}
                               + \frac{\tau_j^2}{4} \eta^j_{\mathrm{ell}, \delta_t}\right)
  \end{gather*}
  with \mbox{$C^*\coloneqq C_\mL C_I \Opnorm{\iL}_{0,2}\,$}, $\sigma_j$ and $\chi_j$ from \autoref{theo-l2-est},
  $C_I$ from \autoref{lem:project},
  $\eta_{S,2}$ from \autoref{lem:green:estimates} and
  \begin{gather*}
    \mu_j \coloneqq \max_{s\in[t_{j-1}, t_j]} \eta_{S,2}(T-s).
  \end{gather*}
\end{theorem}
\begin{proof}
  The proof is very similar to that of \autoref{theo-h1-est}, with
  some details differing.
  First an application of the triangle inequality to~\eqref{eq-res-l2} gives
  \begin{align}\label{eq-res-l2-abs}
      \norm{\left(u-\tilde{u}_h\right)(T)}^2_{0,\Omega} 
        & \le \abs{a\scal{G(t)}{u_0 - u^0_h}}
            + \abs{\int_0^T \dual{\bigl(F - \tilde{F}\bigr)(s)}{G(T-s)}  \D s} \\ 
          \notag
        & \qquad + \abs{\int_0^T a\scal{\bigl(\tilde{R}-\tilde{u}_h\bigr)(s)}{G(T-s)} \D s}
            + \abs{\int_0^T a\scal{\bigl( \tilde{\psi}_h + \pt_s \tilde{u}_h \bigr)(s)}{G(T-s)} \D s}.
  \end{align}
  The terms on the right-hand side need to be bouded separately.
		
  \noindent\emph{(i)} \
  Let \mbox{$P_h$} be the projection operator from \autoref{lem:project}.
  Then, by \eqref{eq-det-u0h-l2}
  \begin{gather*}
    \abs{a\scal{G(T)}{u_0 - u^0_h}}
      = \abs{a\scal{G(T)-P_hG(T)}{u_0 - u^0_h}}
      \le C_\mL \norm{G(T)-P_hG(T)}_{1,\Omega} \norm{u_0 -u^0_h}_{1,\Omega} 
  \end{gather*}
  From \autoref{lem:project}, we infer that
  \begin{gather}\label{eq-h1-err-int-green}
    \norm{G(T)-I_hG(T)}_{1,\Omega} \leq C_I h_{\mathrm{max}} \abs{G(T)}_{2,\Omega}.
  \end{gather}
  The semi norm of $G(T)$ can be bounded using~\eqref{eq:G-2}.
  We get
  \begin{gather*}
    \abs{a\scal{G(T)}{u_0 - u^0_h}}
      \le C_\mL C_I \Opnorm{\iL}_{0,2} h_{\max} \eta_{S,2}(T)
          \norm{\left(u-\tilde{u}_h\right)(T)}_{0,\Omega} \norm{u_0 - u^0_h}_{1,\Omega}.
  \end{gather*}

  \noindent\emph{(ii)} \
  Clearly,
  \begin{gather*}
    \abs{\int_{0}^{T} \dual{\bigl(F - \tilde{F}\bigr)(s)}{G(T-s)} \D s}
      \le \sum_{j=1}^{M} \int_{I_j} \norm{\bigl(F - \tilde{F}\bigr)(s)}_{-1,\Omega}
                         \norm{G(T-s)}_{1,\Omega} \D s.
  \end{gather*}
  Application of eq.~\eqref{eq:G-1} from \autoref{lem:green:estimates} gives
  \begin{gather*}
    \abs{\int_{0}^{T} \dual{\bigl(F - \tilde{F}\bigr)(s)}{G(T-s)} \D s}
      \le \frac{\norm{\left(u-u_h\right)(T)}_{0,\Omega}}{\alpha}
              \sum_{j=1}^{M} \sigma_j \int_{I_j} \norm{\bigl(F-\tilde{F}\bigr)(s)}_{-1,\Omega}.
  \end{gather*}

  \noindent\emph{(iii)} \
  We proceed similarly to \emph{(i)}:
  \begin{align*}
    & \abs{a\scal{\bigl(\tilde{R}-\tilde{u}_h\bigr)(s)}{G(T-s)}}
      = \abs{a\scal{\bigl(\tilde{R}-\tilde{u}_h\bigr)(s)}
                   {\bigl(G-P_hG\bigr)(T-s)}} \\
    & \qquad
      \le C_\mL \norm{\bigl(\tilde{R} - \tilde{u}_h\bigr)(s)}_{1,\Omega}
                \norm{\bigl(G-P_hG\bigr)(T-s)}_{1,\Omega}
      \le C_\mL \norm{\bigl(\tilde{R} - \tilde{u}_h\bigr)(s)}_{1,\Omega}
                C_I h_{\max} \abs{G-P_hG}_{2,\Omega} \\
    & \qquad
      \le C^* h_{\mathrm{max}} \mu_j
              \norm{\bigl(\tilde{R} - \tilde{u}_h\bigr)(s)}_{1,\Omega}
              \norm{\left(u-u_h\right)(T)}_{0,\Omega},
  \end{align*} 
  by~\eqref{eq:G-2}.
  Utilizing \eqref{eq-erro-ell-rec} and \eqref{eq-half-err-ell-rec}  
  we get
  \begin{align*}
    \abs{\int_{0}^{T} a\scal{\bigl(\tilde{R}-\tilde{u}_h\bigr)(s)}{G(T-s)} \D s}
      \le \norm{\left(u -u_h\right)(T)}_{0,\Omega} C h_{\max}
          \sum_{j=1}^{M} \mu_j \left(\tau_j \eta^j_{\mathrm{ell}, 1/2}
                 + \frac{\tau_j^2}{4} \eta^j_{\mathrm{ell}, \delta_t}\right)
  \end{align*}

  \noindent\emph{(iv)} \
  We mimic the procedure in \emph{(iv)} in the proof of \autoref{theo-h1-est}.
  \begin{align*}
     & \int_{I_j} a\scal{\bigl(\tilde{\psi}_h + \pt_s \tilde{u}_h \bigr)(s)}
                        {G(T-s)} \IntD s \\
     & \qquad
       = \int_{I_j} a\scal{\frac{\psi^j_h+\psi^{j-1}_h}{2} + \delta_t u^j_h}
                          {G(T-s)} \D s
           + \int_{I_j} a\scal{\bigl(s-t_{j-1/2}\bigr)\delta_t\psi^j_h}
                              {G(T-s)} \D s \\
     & \qquad
       = \int_{I_j} a\scal{\frac{\psi^j_h+\psi^{j-1}_h}{2} + \delta_t u^j_h}
                          {G(T-s)} \D s
           + \int_{I_j} a\scal{\frac{\bigl(t_j-s\bigr)\bigl(s-t_{j-1}\bigr)}{2}
              \delta_t\psi^j_h}{\iLM G(T-s)} \D s,
  \end{align*}
  where we have used integration by parts and~\eqref{S-exp}.
  Then using the boundedness of $a\scal\cdot\cdot$, \eqref{eq-bound-lm} and
  \autoref{lem-bound-c0}, we obtain
  \begin{align*}
     & \abs{\int_{I_j}
          a\scal{\bigl(\tilde{\psi}_h + \pt_s \tilde{u}_h \bigr)(s)}
                {G(T-s)} \IntD s} \\
     & \qquad
       \le \frac{C_\mL\sigma_j}{\alpha}
           \norm{\left(u-u_h\right)(T)}_{0,\Omega}
           \left(\tau_j \norm{\frac{\psi^j_h+\psi^{j-1}_h}{2}
                          + \delta_t u^j_h}_{1,\Omega}
                   + \min\left\{\frac{\tau_j^2}{4},
                                \frac{C_\mM}{\alpha} \frac{\tau_j^3}{12}
                         \right\}
                     \norm{\delta \psi^j_h}_{1,\Omega}
           \right).
  \end{align*}
  Summing this inequality for $j=1,\dots,M$, and then applying the result
  combined with \emph{(i)} -- \emph{(iv)} to \eqref{eq-res-l2-abs} and
  dividing by $\norm{\left(u-u_h\right)(T)}_{0,\Omega}$, we complete the proof. 
\end{proof}

\begin{remark}
  \emph{(i)} \
  Since \mbox{$H^2(\Omega)\hookrightarrow C(\bar{\Omega})$}
  for \mbox{$\Omega \subset \RR^n$}, \mbox{$n=1,2,3$}, a possible choice
  of $P_h$ is the \mbox{\textsc{Lagrange}} interpolation operator.
		
  \noindent\emph{(ii)} \
  We point out, that the constant $C_I$ is independent of $G$ and thus also
  of $u$.
  Discussions on the value of $C_I$ for \mbox{$\Omega \subset \RR$} can be
  found in \cite[§3.3]{schwab1998finite},
  while the case \mbox{$\Omega \subset \RR^2$} is studied, e.\,g.,
  in \cite[§9.2]{arendt2018partielle}.
\end{remark}

\section{A numerical Example}\label{sec-num-ex}
We consider the following pseudo-parabolic equation
\begin{align*}
-u_{xxt} + \left(5x+6\right)u_t - u_{xx} -\E^{-x}u
&= \E^{-4t} + \cos\left(\pi(x+t)^2\right), \quad \text{on } (-1,1) \times (0,1], \\
\intertext{subject to homogeneous Dirichlet boundary conditions and
       the initial condition}
u(x,0) &= \sin x\pi, \quad \text{for } x\in[-1,1].
\end{align*}
The assumptions of the analysis hold with
\begin{gather*}
  C_\mL = 11\, , \quad \alpha=1\, , \quad C_\mM = \E \quad \text{and} \quad \gamma = -\E.
\end{gather*}
The exact solution of this problem is unknown.
Instead, a reference solution is computed using a spectral \textsc{Galerkin}
method in space combined with the dG$(2)$ method in time,
which is of order~5.

The approximations \mbox{$u^2_h, \dots, u^M_h$} are computed using a conforming
$\mathbb{P}_k$-FEM in space and the BDF-2 method in time.
The first approximation, $u^1_h$, is computed using the backward \textsc{Euler}
method.
Uniform meshes with $M$ mesh intervals are used both in time and in space.
In order to balance spatial and temporal accuracy, we chose \mbox{$k=1$} when
studying the error estimator $\eta_{L^2}$, and \mbox{$k=2$} when
studying $\eta_{H^1}$.

For the elliptic estimator $\eta_{\mathrm{ell}}^j$ we take a somewhat simplified
approach.
We approximate $R^j$, the solution of~\eqref{eq-ell-def}, by $\hat{u}_h^j$ which is
computed using the $\mathbb{P}_{k+1}$-FEM, i.\,e., a method of higher order.
Then we estimate
\begin{gather*}
  \norm{R^j - u^j_h}_{1,\Omega} \approx
  \norm{\hat{u}^j_h - u^j_h}_{1,\Omega} \eqqcolon \eta_{\mathrm{ell}}^j.
\end{gather*}
It has to be noted, that this does in general not give an upper error bound,
however it give an estimate that is asymptotically exact as \mbox{$h_{\max}\to0$}.
Similarly, we choose \mbox{$\eta^j_{\mathrm{ell},1/2}$}
and \mbox{$\eta^j_{\mathrm{ell}, \delta_t}$}. 

The component $\eta_F$ of the error estimators has to be approximated too.
We use the approach outlined in \autoref{rem:eta-f}, by using \textsc{Simpson}'s
rule.
Also, the norms \mbox{$\norm{\cdot}_{0,\Omega}$} and \mbox{$\norm{\cdot}_{1,\Omega}$}
need to be approximated.
For this we employ a \textsc{Gauss-Lobatto} quadrature formula of order 5 on
each spatial interval.

In our experiments our focus is on the efficiency of the error estimators.
For fixed $M$, let $\mathrm{err}_M$ denote the error of our approximation at
final time $T$ and $\mathrm{est}_M$ the corresponding a posteriori error
bound -- either for the $H^1$ norm or form the $L^2$ norm.
Then the efficiency is defined as
\begin{gather*}
  \mathrm{eff}_M \coloneqq \frac{\mathrm{err}_M}{\mathrm{est}_M}\,.
\end{gather*}
Clearly, \mbox{$\mathrm{eff}_M\le 1$}, where values close to $1$ indicate
a high efficiency.
We will also monitor the experimental order of convergence as defined by
\begin{gather*}
  p_M \coloneqq \log_2 \frac{\mathrm{err}_{M/2}}{\mathrm{err}_M}.
\end{gather*}

\paragraph{\boldmath$H^1$-norm.}
\hyperref[tab-bdf2-fem-h1]{Tables \ref{tab-bdf2-fem-h1}} and
\ref{tab-bdf2-fem-com-h1} document the results of our experiments for
the $H^1$-norm estimator.
\autoref{tab-bdf2-fem-h1} contains the number of mesh intervals in time
and space,
the $H^1$-norm error at final time $T$,
the experimental order of convergence $p_M$,
the estimator and the reciprocal of the efficiency $1/\mathrm{eff}_M$.
\begin{table}[h!]
  \centerline{\begin{tabular}{|c|c|c|c|c|}\hline
      $M$ & \mbox{$\norm{\left(u-\tilde{u}_h\right)(T)}_{1,\Omega}$}
          & $p_M$ & $\eta_{H^1}$ & $1/\mathrm{eff}_M$ \\ 
      \hline
        64 & 8.769e-03 & 1.94 & 3.713e-02 & 4.23 \\
       128 & 2.239e-03 & 1.97 & 9.179e-03 & 4.10 \\
       256 & 5.660e-04 & 1.98 & 2.281e-03 & 4.03 \\
       512 & 1.423e-04 & 1.99 & 5.684e-04 & 4.00 \\
      1024 & 3.567e-05 & 2.00 & 1.419e-04 & 3.98 \\
      2048 & 8.923e-06 & 2.00 & 3.543e-05 & 3.97 \\
      4096 & 2.234e-06 & 2.00 & 8.857e-06 & 3.97 \\
      8192 & 5.489e-07 & 2.02 & 2.792e-06 & 5.09 \\\hline
  \end{tabular}}
  \caption{Test results for the $H^1$-norm error estimator}
  \label{tab-bdf2-fem-h1}
\end{table}
We observe convergence of order $2$ -- as expected of the BDF-2 method -- and
a strong correlation of the actual errors and the computed error estimators.
The errors are overestimated by a factor of about $4$.
For our finest discretisation, i.\,e., for \mbox{$M=8192$} we notice a slight
deterioration. This can be attributed to rounding errors from working with
machine numbers.
\autoref{tab-bdf2-fem-com-h1} displays the various components of the error
estimator. 
We observe that the dominant components are $\eta_F$ and $\eta_{\mathrm{ell}}$.

\begin{table}[h!]
  \centerline{\begin{tabular}{|c|c|c|c|c|c|}\hline
      $M$ & $\eta_{\mathrm{init}}$ & $\eta_F$ & $\eta_{\mathrm{ell}}$
          & $\eta_\Psi$ & $\eta_{\delta \psi}$ \\ 
			\hline
	64 & 1.632e-05 & 2.384e-02 & 1.070e-02 & 1.804e-03 & 7.654e-04 \\
       128 & 4.076e-06 & 5.846e-03 & 2.703e-03 & 4.418e-04 & 1.842e-04 \\
       256 & 1.019e-06 & 1.447e-03 & 6.785e-04 & 1.093e-04 & 4.519e-05 \\
       512 & 2.546e-07 & 3.598e-04 & 1.699e-04 & 2.718e-05 & 1.119e-05 \\
      1024 & 6.366e-08 & 8.972e-05 & 4.252e-05 & 6.778e-06 & 2.784e-06 \\
      2048 & 1.592e-08 & 2.240e-05 & 1.063e-05 & 1.692e-06 & 6.943e-07 \\ 
      4096 & 3.979e-09 & 5.597e-06 & 2.660e-06 & 4.228e-07 & 1.734e-07 \\
      8192 & 9.947e-10 & 1.399e-06 & 1.243e-06 & 1.057e-07 & 4.332e-08 \\\hline
  \end{tabular}}
  \caption{Components of the estimator $\eta_{H^1}$}
  \label{tab-bdf2-fem-com-h1}
\end{table}

\paragraph{\boldmath$L^2$-norm.}
We now turn our attention to the $L^2$-norm estimator $\eta_{L^2}$.
This estimator involves a few more constants.
First, \autoref{lem:project} holds true with
\begin{gather*}
  C_I = \frac{3}{2} \sqrt{\frac{1}{2}}\,, \quad
  \text{cf. \cite[Section 3.3]{schwab1998finite}.}
\end{gather*}
The derivation of $\eta_{L^2}$ relies on \autoref{lem:green:estimates}.
It states the existence of constants $M_2$ and $\omega_2$, however these are
not easily determined, as is $\Opnorm{\iL}_{0,2}$.
For demonstration purposes we assume they all have value~$1$.
This solely affects the overestimation of the errors.

\begin{table}[h!]
  \centerline{\begin{tabular}{|c|c|c|c|c|}\hline
    $M$ & \mbox{$\norm{\left(u-\tilde{u}_h\right)(T)}_{0,\Omega}$}
        & $p_M$ & $\eta_{L^2}$ & $1/\mathrm{eff_M}$ \\\hline
     64 & 4.310e-03 & 1.94 & 5.627e-01 & 130.55 \\ 
     128 & 1.102e-03 & 1.97 & 1.401e-01 & 127.16 \\ 
     256 & 2.787e-04 & 1.98 & 3.496e-02 & 125.45 \\ 
     512 & 7.008e-05 & 1.99 & 8.731e-03 & 124.59 \\ 
    1024 & 1.757e-05 & 2.00 & 2.182e-03 & 124.15 \\ 
    2048 & 4.396e-06 & 2.00 & 5.452e-04 & 124.01 \\ 
    4096 & 1.104e-06 & 1.99 & 1.363e-04 & 123.47 \\ 
    8192 & 2.393e-07 & 2.21 & 3.366e-05 & 140.65 \\\hline
  \end{tabular}}
  \caption{Test results for the $L^2$-norm}
  \label{tab-bdf2-fem-l2}
\end{table}

\autoref{tab-bdf2-fem-l2} displays the results of our test computation.
As expected the methods converges with order $2$.
Again we witness a strong correlation of the errors and the a posteriori
error bounds.
The precise values of $\mathrm{eff}_M$ are of course of rather speculative
nature, because we have fixed some unknown constants to be~$1$.


\begin{thebibliography}{99}
  \bibitem{pazy1983semigroups}
    A.~Pazy.
    \newblock {\em {Semigroups of linear operators and applications to partial
    differential equations} }.
    \newblock{Springer-Verlag, New York}, 1983.		
		
  \bibitem{tanabe1979evolution}
    H.~Tanabe.
    \newblock {\em {Equations of evolution}}.
    \newblock {Pitman (Advanced Publishing Program), Boston, Mass.-London}, 1979.
		
  \bibitem{engel2000semigroups}
    K.-J.~Engel~and~R.~Nagel. 
    \newblock {\em {One-parameter semigroups for linear evolution equations}}.
    \newblock {Springer-Verlag, New York}, 2000.
		
  \bibitem{nochetto2003reconstrution}
    Ch.~Makridakis~and~R.~H.~Nochetto.
    \newblock{\em {Elliptic reconstruction and a posteriori error estimates for parabolic problems}}.
    \newblock{SIAM J. Numer. Anal., \textbf{41}}, pp. 1585--1594, 2003.
		
  \bibitem{diestel1977measures}
    J.~Diestel~and~J.~J.~Uhl,~Jr.
    \newblock{\em{Vector measures}}.
    \newblock{American Mathematical Society, Providence, RI}, 1977.
		
  \bibitem{ting1969sobolev}
    T.~W.~Ting.
    \newblock{\em{Parabolic and pseudo-parabolic partial differential equations}}.
    \newblock{J. Math. Soc. Japan, \textbf{21}}, pp. 440--453, 1969.
		
  \bibitem{ting1970sobolev}
    R.~E.~Showalter~and~T.~W.~Ting.
    \newblock{\em{Pseudoparabolic partial differential equations}}.
    \newblock{SIAM J. Math. Anal., \textbf{1}}, pp. 1--26, 1970.
		
  \bibitem{ewing1975sobolev}
    R.~E.~Ewing.
    \newblock{\em{Numerical solution of {S}obolev partial differential
    equations}}.
    \newblock{SIAM J. Numer. Anal., \textbf{12}}, pp. 345--363, 1975.
		
  \bibitem{ewing1978sobolev}
    R.~E.~Ewing.
    \newblock{\em{Time-stepping {G}alerkin methods for nonlinear {S}obolev
    partial differential equations}}.
    \newblock{SIAM J. Numer. Anal., \textbf{15}}, pp. 1125--1150, 1978.
    
  \bibitem{ford1974sobolev}
    W.~H.~Ford~and~T.~W.~Ting.
    \newblock{\em{Uniform error estimates for difference approximations to
    nonlinear pseudo-parabolic partial differential equations}}.
    \newblock{SIAM J. Numer. Anal., \textbf{11}}, pp. 155-169, 1974.
		
  \bibitem{ford1976sobolev}
    W.~H.~Ford.
    \newblock{\em{Galerkin approximations to non-linear pseudo-parabolic partial
    differential equations}}.
    \newblock{Aequationes Math., \textbf{14}}, pp. 271--291, 1976.
		
  \bibitem{thomee1981sobolev}
    D.~N.~Arnold~and~J.~Douglas,~Jr.,~and~V.~Thom\'ee.
    \newblock{\em{Superconvergence of a finite element approximation to the
    solution of a {S}obolev equation in a single space variable}}.
    \newblock{Math. Comp., \textbf{36}}, pp. 53--63, 1981.
		
  \bibitem{tran2005estimation}
    T.~Tran~and~T.-B.~Duong.
    \newblock{\em{A posteriori error estimates with the finite element method of
    lines for a {S}obolev equation}}.
    \newblock Numer. Meth. Part. Diff. Eq., \textbf{21}, pp. 521--535, 2005.
		
  \bibitem{demlow2016error}
    A.~Demlow~and~N.~Kopteva.
    \newblock{\em{Maximum-norm a posteriori error estimates for singularly
    perturbed elliptic reaction-diffusion problems}}.
    \newblock{Numer. Math., \textbf{133}}, pp. 707--742, 2016.
		
  \bibitem{demlow2010error}
    A.~Demlow~and~Ch.~Makridakis.
    \newblock{\em{Sharply local pointwise a posteriori error estimates for
    parabolic problems}}.
    \newblock{Math. Comp., \textbf{79}}, pp. 1233-1262, 2010.
			
  \bibitem{linss2013estimation}
    N.~Kopteva~and~T.~Linss.
    \newblock{\em{Maximum norm a posteriori error estimation for parabolic
    problems using elliptic reconstructions}}.
    \newblock{SIAM J. Numer. Anal., \textbf{51}}, pp. 1494--1524, 2013.
    
  \bibitem{ossadnik2024unified}
    T.~Lin\ss~and~M.~Ossadnik~and~G.Radojev.
    \newblock{\em{A unified approach to maximum-norm {\it a posteriori} error
    estimation for second-order time discretizations of parabolic
    equations}}.
    \newblock{IMA J. Numer. Anal., \textbf{44}}, pp. 1644--1659, 2024.	
		
  \bibitem{grossmann2007numerical}
    Ch.~Großmann~and~~H.-G.~Roos~and~M.~Stynes
    \newblock{\em{Numerical treatment of partial differential equations}}.
    \newblock{Springer Berlin Heidelberg New York}, 2007
		
  \bibitem{arendt2018partielle}
    W.~Arendt~and~K.~Urban
    \newblock{\em{Partielle Differenzialgleichungen: Eine Einführung in analytische und numerische Methoden}},
    \newblock{Springer Berlin Heidelberg}, 2018
	 	
  \bibitem{grisvard1985elliptic}
    P.~Grisvard
    \newblock{\em{Elliptic problems in nonsmooth domains}},
    \newblock{SIAM, Philadelphia, PA}, 2011
		
  \bibitem{schwab1998finite}
    Ch.~Schwab
    \newblock{\em{$p$- and $hp$-finite element methods}},
    \newblock{The Clarendon Press, Oxford University Press, New York}, 1998
\end{thebibliography}
\end{document}